\documentclass[12pt]{amsart}
\usepackage{amssymb,amscd,verbatim}
\usepackage[all]{xy}
\usepackage[colorlinks,linkcolor=blue,citecolor=blue,urlcolor=red]{hyperref}
\usepackage[symbol]{footmisc}

\usepackage{color}

\swapnumbers
\newtheorem{thm}{Theorem}[subsection]
\newtheorem{propose}[thm]{Proposition}
\newtheorem{lemma}[thm]{Lemma}
\newtheorem{cor}[thm]{Corollary}

\theoremstyle{definition}
\newtheorem{defn}[thm]{Definition}
\newtheorem{axiom}[thm]{Axiom}
\newtheorem{remark}[thm]{Remark}
\newtheorem{remarks}[thm]{Remarks}
\newtheorem{example}[thm]{Example}

\newcommand{\N}{\mathbb{N}}
\newcommand{\Z}{\mathbb{Z}}     
\newcommand{\R}{\mathbb{R}} 
\newcommand{\C}{\mathbb{C}}     
     
\newcommand{\Aff}{\mathbb{A}}   

\renewcommand{\d}{{\text{\LARGE $\cdot $}}}

\newcommand{\Spec}{\operatorname{Spec}} 
\newcommand{\Hom}{\operatorname{\underline{\sf Hom}}} 
  
\newcommand{\Ex}{\operatorname{\sf Ex}}      

\newcommand{\Psh}{\operatorname{Pshv}}
\newcommand{\SPsh}{\operatorname{SPshv}}

\newcommand{\Ch}{\operatorname{Ch}}
\newcommand{\Set}{\operatorname{Set}}
\newcommand{\Top}{\operatorname{CTop}}
\newcommand{\SSet}{\operatorname{SSet}}
\newcommand{\Sm}{\operatorname{Sm}}
\newcommand{\Sch}{\operatorname{Sch}}

\newcommand{\RMod}{\text{$R$-{Mod}}}
\newcommand{\Rmod}{\text{$R$-{mod}}}

\newcommand{\iRMod}{\text{\rm $R$-{\underline{Mod}}}}
\newcommand{\iRmod}{\text{\rm $R$-{\underline{mod}}}}

\newcommand{\Ind}{\operatorname{Ind}}

\newcommand{\im}{\operatorname{Im}}        
\renewcommand{\ker}{\operatorname{Ker}}  
\newcommand{\coker}{\operatorname{Coker}} 

\newcommand{\tr}{{\operatorname{tr}}}        
\newcommand{\qi}{\text{q.i.}\,}      

\newcommand{\by}[1]{\stackrel{#1}{\rightarrow}}
\newcommand{\longby}[1]{\stackrel{#1}{\longrightarrow}}

\renewcommand{\tilde}{\widetilde}
\newcommand{\df}{\mbox{\,${:=}$}\,}
\newcommand{\ie}{{\it i.e. }}
\newcommand{\cf}{{\it cf. }}
\newcommand{\eg}{{\it e.g. }}

\newcommand{\Nis}{{\operatorname{Nis}}}

\newcommand{\into}{\hookrightarrow}
\newcommand{\onto}{\mbox{$\to\!\!\!\!\to$}}
\newcommand{\qto}{\rightleftarrows}

\renewcommand{\lim}{\varprojlim}

\newcommand{\boxtensor}{\def\boxtimesten{\Box\kern-7.59pt\raise1.2pt
\hbox{$\times$} }}                                  

\newcounter{elno}                   

\newcommand{\cA}{\mathcal{A}}

\newcommand{\cC}{\mathcal{C}}
\newcommand{\cD}{\mathcal{D}}
\newcommand{\cE}{\mathcal{E}}

\newcommand{\cP}{\mathcal{P}}

\newcommand{\cR}{\mathcal{R}}

\renewcommand{\phi}{\varphi}
\renewcommand{\epsilon}{\varepsilon}

\title{On topological motives}
\author{Luca Barbieri-Viale}
\address{Dipartimento di Matematica ``F. Enriques'', Universit{\`a} degli Studi di Milano\\  Via C. Saldini, 50\\ I-20133 Milano\\ Italy }
\email{luca.barbieri-viale@unimi.it}
\begin{document}
\begin{abstract}
Following Eilenberg-Steenrod axiomatic approach we construct the universal ordinary homology theory for any homological structure on a given category by representing ordinary theories with values in abelian categories. For a convenient category of spaces we then obtain a universal abelian category which can be actually described for CW-complexes as the category of hieratic modules.
\end{abstract}
\maketitle
\tableofcontents
\section*{Introduction}
We introduce homological categories defining ordinary and extraordinary  homology theories with values in abelian categories; by resuming and generalising the Eilenberg-Steenrod axiomatic approach \cite{ES} (see also \cite[Sect.\,2.3]{Ha}) we show that the setting treated in \cite{UCTI} is sufficiently malleable to be used to construct universal ordinary homologies on any homological category. For suitable topological categories, the context of ordinary theories or theories with transfers is sufficient for ``topological motives'' via universal ordinary and extraordinary homologies, in addition to the universal homotopy theories \cite{DU}. For any commutative unitary ring $R$ we get that ordinary $R$-linear topological motives associated with CW-complexes are hieratic $R$-modules (\ie the indization of Freyd's free $R$-linear abelian category on a point); remark that this result implies that topological $R$-linear Nori motives are given by $R$-modules if the ring $R$ is coherent, \ie the universal abelian category generated by usual singular homology on CW-complexes is a canonical quotient of ordinary topological motives (see \cite{HMS} for Nori's construction).\\

Let $\cC$ be a category together with a distinguished subcategory  and let $\cC^{\square}$ be a category of pairs (as in \cite{UCTI}): $\Hom (\cC, \cA)$ and $\Hom (\cC^{\square}, \cA)$ are the categories of homologies on $\cC$ and relative homologies on $\cC^{\square}$ with values in an $R$-linear abelian category $\cA$. 
We introduced universal abelian categories $\cA(\cC)$ and $\cA_\partial(\cC)$ generated by  homology and relative homology theory, respectively. These are given by 
$$H =\{H_i\}_{i\in \Z } : \cC \to  \cA (\cC) \hspace{0.2cm}\text{and}\hspace{0.2cm} H =\{H_i\}_{i\in \Z } : \cC^{\square}\to  \cA_\partial (\cC)$$
family of functors satisfying basic axioms, see \cite[Thm.\,2.1.2 \& 2.2.4]{UCTI}. In particular, any relative homology $H'\in \Hom (\cC^{\square}, \cA)$ is classified by $r_{H'}: \cA_\partial(\cC)\to \cA$ a unique (up to natuarl isomorphism) exact functor such that $r_{H'}(H) =H'$.

Adding axioms to these ``free'' theories we obtain decorations $\cA^\dag (\cC)$ and $\cA_\partial^\dag (\cC)$. Adding the point axiom $\dag ={\rm point}$ (see \cite[Axiom 2.3.2]{UCTI}) we obtain $\cA_\partial^{\rm point}(\cC)$ which is a quotient of $\cA_\partial (\cC)$ (see \cite[Prop.\,2.1.7 \& 2.3.3]{UCTI}). Furthermore, adding excision and homotopy invariance axioms to the relative theory with the point axiom we obtain more ``relations'' providing the universal Eilenberg-Steenrod abelian category $\cA_\partial^{\rm ord}(\cC)$ as a quotient of $\cA_\partial^{\rm point}(\cC)$. 
For  example,  for $\cC={\bf 1}$, the point category and $\cC={\bf 2}$, the two objects category determined by the ordinal $2\in \N$, we get 
$$\cA^{\rm point} ({\bf 1})\cong \cA^{\rm point}_\partial ({\bf 2})\cong \iRmod$$ where the indization 
$\Ind \iRmod = \iRMod$ is the Grothendieck category of hieratic $R$-modules: see \cite[Example 2.1.8 \& 2.3.4]{UCTI} and \cite[Def.\,1.3.3]{UCTI} for the definition and the universal property of hieratic $R$-modules. In particular, for $R$ coherent we have an exact $R$-linear quotient functor $$r_R: \iRmod\to\Rmod$$ to the abelian category of finitely presented $R$-modules such that $r_R(|R|)=R$ where $|R|$ is the universal object.\\

Actually, let $\cC^\square$ be equipped with a collection $\Xi$ of excisive squares $\epsilon : (Y, Z)\to (W, X)$ in $\cC^\square$ (see Definition \ref{excisive}) and a cylinder functor $I$, \eg induced by an interval object $I\in \cC$, providing $I$-homotopies (see Definitions \ref{interval} and \ref{Ihomotopic}). Say that $(\cC^\square,  \Xi, I)$ is a \emph{homological structure} on the underlying category $\cC$ together with a selection of point objects (see Definition \ref{ES}); we clearly have a corresponding notion of \emph{homological functor} (see Definition \ref{homdef}). Consider the $I$-homotopy invariant homologies (see Axiom \ref{hinvax}) satisfying excision (see Axiom \ref{excisax}): call these relative homologies \emph{ordinary} homologies if they also satisfy the point axiom (see Definition \ref{ES}). 

We then get the \emph{universal ordinary homology} $H$ with values in the abelian category $\cA_\partial^{\rm ord}(\cC)$ which is a quotient of the ``free" one $\cA_\partial(\cC)$ and representing the subfunctor $\Hom_{\rm ord} (\cC^{\square}, -)\subseteq \Hom (\cC^{\square}, - )$ of ordinary theories. Similarly, further imposing additivity (see \cite[Axiom 2.3.8]{UCTI} and \cf \cite{M}), there is a universal Grothendieck abelian category $\cA_\partial^{\rm Ord}(\cC)$ which is a Grothendieck quotient of $\Ind \cA_\partial^{\rm ord} (\cC)$ (see Theorem \ref{univES}):  the universal ordinary homology is given by composition in the following commutative diagram
$$\xymatrix{\cC^{\square}\ar[r]^-{}  \ar[dr]_-{H} & \ar[r]^-{}  \cA_\partial (\cC) \ar[d]^-{} &\Ind \cA_\partial (\cC) \ar[d]^-{} \\
&  \cA_\partial^{\rm ord} (\cC) \ar[r]^-{} & \cA_\partial^{\rm Ord} (\cC)}$$ 
where the vertical functors are quotients. We also get the model category $\Ch(\cA_\partial^{\rm Ord} (\cC))$ given by chain complexes and a canonical Quillen pair $U\cC^\square\qto  \Ch(\cA_\partial^{\rm Ord} (\cC))$ induced by the universal homotopy theory \cite{DU}  (see Theorem \ref{hothom}).\\

Moreover, if $\cC$ is provided with a cosimplicial object, picking singular homologies with values in Grothendieck categories, we obtain universal singular homology
$$H^{\rm Sing}=\{H^{\rm Sing}_i\}_{i\in \N}: \cC^\square\to \cA^{\rm Sing}_\partial(\cC)$$ representing $\Hom_{\rm Sing} (\cC^{\square}, -)\subseteq \Hom_{\rm Ord} (\cC^{\square}, -)$ (see Theorem \ref{unising}) forcing singular homologies to be ordinary.  Therefore, we get the following commutative diagram
$$\xymatrix{\cC^{\square}\ar[r]^-{}  \ar[dr]_-{H^{\rm Sing}} & \ar[r]^-{} \Ind  \cA_\partial (\cC) \ar[d]^-{} &\cA_\partial^{\rm Ord} (\cC) \ar[d]^-{r_{H^{\rm Sing}}} \\
&  \iRMod \ar[r]^-{} & \cA_\partial^{\rm Sing} (\cC)}$$ 
where $\cA_\partial^{\rm Sing} (\cC)$ is a quotient of  $\iRMod$. If singular homologies are ordinary with respect to the homological structure we have that 
$\iRMod\cong\cA^{\rm Sing}_\partial(\cC)$ (see  Lemma \ref{relsing}).\\

For $\cC = \Top$ a convenient category of topological spaces endowed with the standard homological structure $(\Top^\square,\Xi, I)$ (see Definition \ref{stop}) we thus obtain a  Grothendieck category of ordinary topological motives $\cA^{\rm Ord}_\partial(\Top)$ along with 
 $$H: \Top^\square\to \cA^{\rm Ord}_\partial(\Top)$$
the universal Eilenberg-Steenrod or ``motivic'' homology. 
 By the way we also have versions with transfers (see Remark \ref{trans}).
Note that usual singular homology with $R$ coefficients $H^{\rm Sing}(-;R):\Top^{\square}\to \RMod$  yields the following picture 
$$\xymatrix{
&\cA_\partial^{\rm Ord} (\Top)\ar[dr]^-{r_{H^{\rm Sing}}}\ar[d]&\\
\Top^{\square}\ar[r]^-{H^{\rm Sing}} \ar[dr]_-{H^{\rm Sing}(-;R)} \ar[ur]^-{H} & 
\iRMod \ar@/_1.2pc/[u]_{!_\partial} \ar[r]^-{\simeq} \ar[d]^-{\Ind r_R} &\cA_\partial^{\rm Sing} (\Top) \ar[dl]^-{\ \ r_{H^{\rm Sing}(-;R)}} \\
&  \RMod & }$$ 
where the canonical exact functor $\Ind r_R$ is a quotient if $R$ is coherent and the canonical functor $!_\partial$  is a splitting functor. 

In general, a concrete description of these universal categories $\cA^{\rm Ord}_\partial(\cC)$ is missing. However, under canonical ``cellular" conditions on the homological structure $(\cC^\square,  \Xi, I)$ (see Definition \ref{ordcell}) the ordinary homologies are represented by a corresponding  ``cellular" complex (see Proposition \ref{cellularity}). For topological CW complexes we then get that ordinary topological motives are just given by $\iRMod$ (see Theorem \ref{ordCW}) and topological Nori motives are given by $\RMod$ if $R$ is coherent (see Corollary \ref{NoriCW}).

\subsubsection*{Acknowledgements} 
I would like to thank Joseph Ayoub for brainstorming on matters treated in this paper and, in particular, on the Theorem \ref{ordCW}. Thanks to UZH \& FIM--ETH Z\"urich  for providing support, hospitality and excellent working conditions. 

\subsubsection*{Notations and assumptions} We shall adopt the same notations and conventions as in \cite{UCTI}. Let $R$ be a commutative unitary ring. Let $\iRMod$ be the universal category of hieratic $R$-modules in \cite{UCTI}. 
 Let $\cC$ be any (small) category and let $\cC^\square$ be the category whose objects are morphisms of a distinguished subcategory of $\cC$ and morphisms are commutative squares of $\cC$. 
\section{Universal Eilenberg-Steenrod homology}
\subsection{Excision and Mayer-Vietoris}
For each $W\in \cC$ we consider a set $\cE_W$ of commutative squares of $\cC$
$$\xymatrix{Z\ar[r]^-{}  \ar[d]_-{} & Y \ar[d]^-{} \\
X \ar@{>}[r]^-{} & W}$$ 
where the horizontal morphisms are distinguished and we assume that the squares in  $\cE_W$ are stable under  isomorphisms. Equivalently, such a square is a morphism $\epsilon : (Y, Z)\to (W, X)$ in $\cC^\square$ and the  assumption  is that for each $W\in \cC$ the set $\cE_W$ contains the identity of $(W, X)$ in  $\cC^\square$  for any distinguished morphism $X\to W$ and the isomorphism class of $\epsilon$ in the category of cubes $\cC^\blacksquare$ given by $\cC^\square$, see \cite[Remark 2.2.2 b)]{UCTI}. 
\begin{defn}\label{excisive}
Say that such a set of squares $\cE_W$ is a set of \emph{excisive squares over} $W$.  We say that $\cC^\square$ is provided with a collection $\Xi = \{\cE_W\}_{W\in \cC}$ of excisive squares if a set of excisive squares  $\cE_W$  over $W$ is given for any $W\in \cC$. Say that $\Xi ={\rm Iso}$ is \emph{initial} (or trivial or coarse) if  $\cE_W$ contains only isomorphism classes of objects of  $\cC^\square$ for any $W\in \cC$. Say that $\Xi =\cC^\square$ is \emph{final} (or finest or discrete).
\end{defn}
 \begin{axiom}[Excision] \label{excisax} 
 For $\cC^\square$ equipped with a collection $\Xi$ of excisive squares we say that
 $H\in \Hom (\cC^\square, \cA)$ satisfies \emph{excision} if  we have isomorphisms $\epsilon_i : H_i (Y,Z)\cong H_i  (W, X)$ induced by $\epsilon : (Y, Z)\to (W, X)$ for all excisive squares over $W\in \cC$  and $i \in \Z$.
 \end{axiom} 
 \begin{remark}\label{exctriv}
 We point out that all relative homologies satisfy excision with respect to the initial collection  $\Xi ={\rm Iso}$ and no non trivial homology satisfies excision if $\Xi$ is final and $\cC$ has an initial object. If $\Xi$ contains $(Y, Y)\to (X, Y)$ for $Y\neq 0$ where $0$ is strictly initial we get that $H_i  (X,Y)=0$ for all $i\in\Z$ if $Y\neq 0$ so that $H_i (Y)\cong H_i  (X)$ for all $(X, Y)\in \cC^\square$. Moreover, if $\Xi$ just contains $(X, 0)\to (X, Y)$ then $H_i (X)\cong H_i  (X,Y)$ for all $i\in \Z$ which implies that $H$ is trivial if this holds for all $(X, Y)\in \cC^\square$. 
 \end{remark}
 As expected, excision implies $cd$-exactness (\cf \cite{Vcd} and \cite[Def.\,2.52]{BV}). 
 \begin{propose}[Mayer-Vietoris sequence]\label{MV}
 Let $\cC^\square$ be with a collection $\Xi$ of excisive squares and an initial object. If $H\in \Hom (\cC^\square, \cA)$ satisfies excision then there is a  long exact sequence 
$$\cdots \to H_i(Z)\to H_i(Y)\oplus H_i(X) \to H_i(W)\to H_{i-1}(Z)\to\cdots$$
 for each $ (Y, Z)\to (W, X)$  excisive square. The sequence is natural with respect to cubes which are morphisms of excisive square. 
 \end{propose}
 \begin{proof} Standard: the excisive square $ (Y, Z)\to (W, X)$ induces the following commutative square in 
 $\cC^\square$
  $$\xymatrix{(Z,0)\ar[r]^-{}  \ar[d]_-{} & (Y,Z) \ar[d]^-{} \\
(X,0)\ar[r]^-{} & (W,X)}$$
which is a particular $\partial$-cube; by naturality of $\partial$-cubes we get a morphism from the long exact sequence of the pair $(Y, Z)$ to the sequence of the pair $(W, X)$ which yields the claimed Mayer-Vietoris sequence by diagram chase. 
Naturality follows from naturality with respect to $\partial$-cubes. \end{proof}
If $\cC$ has an initial object $0$, considering $Z=0$ in the long exact sequence of Proposition \ref{MV} we obtain that $H_i(Y)\oplus H_i(X)\by{\simeq}H_i(W) $ if  
$$\xymatrix{0\ar[r]^-{}  \ar[d]_-{} & Y \ar[d]^-{} \\
X \ar[r]^-{} & W}$$ 
is an excisive square. In particular, if the coproduct exists we may set $W=X\coprod Y$ and we obtain: 
\begin{propose} \label{exprod}
Let $\cC$ be with finite coproducts such that the canonical morphisms to the coproduct are distinguished.
The category $\cC^\square$ has excisive coproducts for a relative homology if and only if the relative homology is (finitely) additive (compare with \cite[Axiom 2.3.8]{UCTI}).
\end{propose}
\begin{proof} For $(X, Y)\in \cC^\square$ consider the distinguished morphisms $u :X\to X\coprod Y$, $\upsilon : Y \to X\coprod Y$ and the induced  $\kappa: (Y, 0) \to (X\coprod Y, X)$ in $\cC^\square$. We then obtain the following diagram with exact rows
$$\xymatrix{0\ar[r]^-{} & H_i(X)\ar[r]^-{}\ar[dr]^-{u_i}  \ar[d]^-{||} & H_i(X)\oplus H_i(Y) \ar[d]^-{\gamma_i}\ar[r]^-{\pi_i} & H_i(Y) \ar[r]^-{}\ar[dl]_-{\upsilon_i} \ar[d]^-{\kappa_i} &  0 \\
\cdots\ar[r]^-{} & H_i(X)\ar[r]^-{\alpha_i}  & H_i(X\coprod Y) \ar[r]^-{\beta_i} & H_i(X\coprod Y, X) \ar[r]^-{}  &  \cdots}$$
where  the top row is given by the  canonical inclusion/projection for the direct sum (it is the Mayer-Vietoris sequence of the identity square). By naturality of the exact sequences of the pair compared via $\kappa$ we have that $\kappa_i =  \beta_i\upsilon_i$. 

Now the left square commutes, \ie $\alpha_i=u_i $, and the commutativity of the right square, \ie $\kappa_i\pi_i = \beta_i\gamma_i$ where $\pi_i $ is the canonical projection, holds true since  $\beta_iu_i =0$ and $\gamma_i = u_i +\upsilon_i$ so that we have $\beta_i\gamma_i (x, y) = \beta_iu_i(x) + \beta_i\upsilon_i (y)= \kappa_i(y)=\kappa_i\pi_i(x, y)$ for $(x, y)\in H_i(X)\oplus H_i(Y) $. Thus we obtain a map of exact sequences.
Then $\gamma_i: H_i(X)\oplus H_i(Y) \cong H_i (X\coprod Y)$ is an iso for all $i\in \Z$ if and only if $\kappa_i: H_i(Y)\cong H_i (X\coprod Y, X)$ is an iso for all $i\in \Z$, as it follows from an easy diagram chase. 
\end{proof}
\begin{example} \label{exalb}
Let $\cC = \cA$ be an abelian category with monos as distinguished subcategory and excisive coproducts. Moreover, for each mono $B\into A$  consider the following excisive square
  $$\xymatrix{B\ar[r]^-{}  \ar[d]_-{} & A\ar[d]^-{} \\
0\ar[r]^-{} & A/B}$$
and collect in $\Xi_\partial$ the isomorphism class of these together with the coproduct squares. A relative homology is excisive with respect to $\Xi_\partial$ if and an only if is a $\partial$-homology in the sense of \cite[Def.\,3.1.1]{UCTI}. 
\end{example}

\subsection{Homotopy invariance}
We shall consider a ``cylinder functor'' and an ``interval" in the weakest possible sense. For Voevodsky interval see \cite[\S 2.2]{Vh}. Recall that Kan \cite{KA} introduced cylinder functors and Grandis \cite{Gra} has shown that structured cylinder functors are a basis for homotopical algebra.
\begin{defn}\label{interval}
 A \emph{cylinder functor} is an endofunctor $ I : \cC^\square\to \cC^\square$ together with a natural transformation $ p: I (-) \to (-)$.  An \emph{interval} of $\cC$ is an object $I\in \cC$  such that the product $X\times I \in \cC$ exists for all objects $X\in \cC$ compatibly with the distinguished morphisms, \ie the following 
$$(X, Y)\leadsto I (X, Y) \df  (X\times I , Y\times I ): \cC^\square \to \cC^\square$$
is functorial.
\end{defn}
Actually, for any interval object $I\in \cC$, we have that  $p: I (X, Y)\to (X, Y)$ given by
$$p\df (\pi_X, \pi_Y) : (X\times I , Y\times I )\to (X,Y)$$ the canonical projections is clearly a natural transformation of functors in the variable $(X, Y)\in \cC^\square$. This is providing the cylinder functor induced by an interval object.
\begin{axiom}[$I$-invariance]  \label{Invax} 
For a cylinder functor $(I, p)$ on $\cC^\square$ the relative homology $H$ satisfies \emph{$I$-invariance} if  $p_i: H_i (I (X, Y))\cong  H_i  (X, Y)$ is an isomorphism for all $(X, Y)\in \cC^\square$ and $i \in \Z$.
\end{axiom} 
If $p$ is a natural equivalence then every relative homology is trivially $I$-invariant.  In general, $H$ is $I$-invariant if and only if $p: I (X, Y)\to(X, Y)$ is excisive.

Moreover, assume that $\cC\into \cD$ is a (full) subcategory inducing a faithful functor  $\cC^\square\into \cD^\square$ between distinguished subcategories. In general, for any cylinder functor $ I : \cD^\square\to \cD^\square$ we have:
\begin{defn}\label{Ihomotopic}
 Say that two parallel maps $\gamma, \delta :  (X,Y)\to (X',Y')\in \cC^\square$  are \emph{$I$-homotopic} if  there are  $\varsigma,\  \sigma: (X,Y)\to I (X, Y)$ and $\xi : I (X, Y)\to (X',Y')$, morphisms of  $\cD^\square$,  such that  $p\varsigma=p\sigma= id_{(X,Y)}$, $\gamma = \xi  \varsigma$ and $\delta =\xi \sigma$. Denote $\gamma \approx_I \delta :(X,Y)\to (X',Y')$ two $I$-homotopic morphims of $\cC^\square$.
 \end{defn}
 \begin{remark}
 Note that $\varsigma\approx_I \sigma$ for all morphisms $\varsigma,\  \sigma: (X,Y)\to I (X, Y)$ such that $p\varsigma=p\sigma= id_{(X,Y)}$.  For $\cC\neq \cD$, assume, for example, that $\cD$ has a terminal object $1$, pick an interval $I$ of the category $\cD$, see Definition \ref{interval}, such that $I\notin \cC$, eventually, and morphisms $\iota, \nu : 1\to I$. Note that we then have that  $ \iota_X ,  \nu_X : X\to X \times I $ for every $X\in \cC$ and we get morphisms $ \iota_{(X,Y)} ,  \nu_{(X,Y)} : (X, Y) \to (X\times I , Y\times I )$ in $\cD^\square$ such that $p \iota_{(X,Y)}=p \nu_{(X,Y)}= id_{(X,Y)}$. We then get a Kan cylinder functor on  $\cD^\square$ (\cf \cite[\S 2]{KA}) induced by the interval $I$ of the category $\cD$ and $\varsigma=\iota_{(X,Y)}$,  $\sigma = \nu_{(X,Y)}$ induced by $\iota, \nu : 1\to I$ can be taken in Definition \ref{Ihomotopic}. In particular, if  $\cC=\cD$  has an interval and
$\iota, \nu : 1\to I\in \cC$ then we get the usual $I$-homotopy. 
\end{remark}
In general, for such $\cC^\square\into \cD^\square$ and $ I : \cD^\square\to \cD^\square$ a cylinder functor we may also consider a generalised notion of $I$-homotopy invariance.
 \begin{axiom}[Homotopy invariance]  \label{hinvax} 
 We say that $H\in \Hom (\cC^\square, \cA)$ is \emph{$I$-homotopy invariant} if we have that $\gamma_i=  \delta_i: H_i (X, Y)\to H_i  (X', Y')$ for all   $\gamma \approx_I \delta: (X,Y)\to (X',Y')\in \cC^\square$ $I$-homotopic and $i \in \Z$.  \end{axiom} 
Let $\Hom_{I-\rm hi} (\cC^\square, - )\subset \Hom (\cC^\square, - )$ be the subfunctor of $I$-homotopy invariant  homologies. Let $\Hom_{I} (\cD^{\square}, - )$ be the 2-functor of $I$-invariant relative homologies.
 \begin{lemma} \label{homotopy}
For any cylinder functor $ I : \cD^\square\to \cD^\square$ and $H: \cD^\square\to \cA$, by restriction along $\cC^\square\into \cD^\square$, we have that  $I$-invariance over $\cD^\square$ implies  $I$-homotopy invariance over $\cC^\square$. If $\cC=\cD$ then $\Hom_{I} (\cC^\square, - ) \subseteq \Hom_{I-\rm hi} (\cC^\square, -)$ is a subfunctor. 
 \end{lemma}
\begin{proof} For $H\in \Hom_{I} (\cD^{\square}, \cA)$ we have $p_i \varsigma_{ i}=p_{i} \sigma_i= id_{i}: H_i(X,Y)\to H_i(X, Y)$ and $p_i$ isomorphism whence
$\varsigma_i= \sigma_i : H_i(X, Y)\to H_i(I(X, Y))$ for all $(X,Y)\in \cD^{\square}$. In particular, for $ \gamma, \delta : (X,Y)\to (X',Y')\in \cC^\square$  we obtain $\gamma_i=  \delta_i$ if $\gamma \approx_I \delta$. 
\end{proof}
Recall that Voevodsky \cite[\S 2.2]{Vh} call an interval of $\cC$ with final object $1$, an interval $I^+$ together with morphisms $m:I^+\times I^+\to I^+$ and $i_0,i_1:1\to I^+$ such that 
$$
m(i_0\times id) = m(id\times i_0) = i_0p\hspace*{0.5cm}
m(i_1\times id) = m(id\times i_1) = id
$$
where $p:I^+\to 1$ is the canonical morphism, $id$ is the identity of $I^+$ and $i_0 \times id:I^+\cong 1\times I^+\to I^+\times I^+$, etc.
\begin{lemma} \label{hinv=hi} If a category $\cC$ has an interval object $I=I^+ $ in the sense of Voevodsky then $I$-homotopy invariance is equivalent to $I$-invariance.
\end{lemma}
\begin{proof} It is a standard argument, \cf \cite[Lemma 3.8.1]{BV}.
\end{proof}
 \begin{propose}\label{hominv}
If $\cC^\square\into \cD^\square$ with a cylinder functor $I$ on $\cD^\square$ then $\Hom_{I-\rm hi} (\cC^{\square}, - )$ is representable, \ie there is a universal $I$-homotopy invariant relative homology  
  $H = \{ H_i\}_{i\in \Z} :\cC^{\square}\to \cA_\partial^{I-\rm hi}(\cC)$
with values in $ \cA_\partial^{I-\rm hi}(\cC)$ a quotient of $ \cA_\partial(\cC)$. Moreover, $\Hom_{I} (\cD^{\square}, - )$ is representable; if $\cC=\cD$ then there is a universal $I$-invariant relative homology  
$H = \{ H_i\}_{i\in \Z} :\cC^{\square}\to \cA_\partial^{I}(\cC)$
with values in $ \cA_\partial^{I}(\cC)$ a quotient of $\cA_\partial^{I-\rm hi}(\cC)$. If  $I\in \cC$ is an interval in the sense of Voevodsky then 
$\cA_\partial^{I-\rm hi}(\cC)\cong\cA_\partial^I(\cC)$.
\end{propose}
\begin{proof}
The category $ \cA_\partial^{I-\rm hi}(\cC)$  is the quotient of  $ \cA_\partial(\cC)$ by the thick subcategory generated by $\im (\gamma_i-  \delta_i)$  for $\gamma_i, \delta_i: H_i (X, Y)\to H_i  (X', Y')$ in $\cA_\partial(\cC)$ all parallel $I$-homotopic morphisms  $\gamma \approx_I \delta: (X,Y)\to (X',Y')$  and $i \in \Z$. The category $ \cA_\partial^{I}(\cD)$  is the quotient of  $ \cA_\partial(\cD)$ by the thick subcategory generated by  the kernels and cokernels of $p_i: H_i (I (X, Y))\to H_i  (X, Y)$ for $i\in \Z$. For $\cC = \cD$ the other claims follow from Lemmas \ref{homotopy}- \ref{hinv=hi}.
 \end{proof}
 As usual, we have that a morphism $\gamma : (X,Y)\to (X',Y')$ is an $I$-homotopy equivalence if there is $\delta: (X',Y')\to (X,Y)$ such that $\gamma  \delta \approx_I id_{(X',Y')}$ and  $\delta\gamma   \approx_I id_{(X,Y)}$. This implies that $\gamma_i :  H_i (X, Y)\by{\simeq} H_i  (X', Y')$ is an isomorphism for any $H\in \Hom_{I-\rm hi} (\cC^\square, \cA)$ homotopy invariant homology. 
 \begin{defn}\label{univgenexc}
Say that $\gamma :  (X,Y)\to (X',Y')\in \cC^\square$ is  \emph{universally homologically invertible} if  $\gamma_i :  H_i (X, Y)\by{\simeq} H_i  (X', Y')$ is invertible in $\cA_\partial (\cC)$ where $H : \cC^\square \to  \cA_\partial (\cC)$ is the universal relative homology. Say that $\gamma :  (X,Y)\to (X',Y')\in \cC^\square$ is \emph{universally $I$-homotopically invertible} if  $\gamma_i :  H_i (X, Y)\by{\simeq} H_i  (X', Y')$ is invertible in $\cA_\partial^{I-\rm hi}(\cC)$ where $H : \cC^\square \to  \cA_\partial^{I-\rm hi} (\cC)$ is the universal $I$-homotopy invariant homology. 
For $\cC^\square$ endowed with a collection $\Xi$ of excisive squares we say that a morphism $\eta : (X, Y)\to (X',Y')\in \cC^\square$ is a \emph{$(\Xi , I)$-generalised excision} if $\eta$ is a finite composition of excisions squares in $\Xi$ and universally $I$-homotopically invertible morphisms.
\end{defn}
A stronger notion of generalised excision is considered in \cite[Chap.\,IV, Def. 9.2]{ESF}
Clearly,  $\gamma$ is universally homologically (resp.\,homotopically) invertible if and only if  $\gamma_i :  H_i (X, Y)\by{\simeq} H_i  (X', Y')$ is invertible in $\cA$ for $H : \cC^\square \to  \cA$ any relative (resp.\,homotopi\-cally invariant) homology. For example, an $I$-homotopy equivalence is universally $I$-homotopically invertible.
\begin{remark}\label{invtriv}
Note that every category $\cC$ with a final object $1$ is endowed with the trivial Voevodsky interval object $I =1$  where $\iota = \nu =id_{1}$ in such a way that $(X, Y) \leadsto 1(X, Y) = (X,Y) : \cC^\square \to \cC^\square$ is the identity. For the interval $1$, since $\varsigma= \sigma =id_{1}$, we have that  $\gamma \approx_1 \delta$ if and only if  $\gamma = \delta$. 
Therefore, a $1$-homotopy equivalence is an isomorphism of $\cC^\square$, all relative homologies are $1$-invariant, universal $1$-homotopical homology coincide with universal relative homology and we have $\cA_\partial (\cC)\cong\cA_\partial^{1-\rm hi}(\cC)\cong\cA_\partial^1(\cC)$. 

On the other hand, if $\cC$ is a category with a strictly initial object $0$ then also $I=0$ is an interval where $(X, Y) \leadsto 0(X, Y) = (0,0) : \cC^\square \to \cC^\square$. For the interval $0$ we have that $\gamma$ is never $0$-homotopic to $\delta$, unless $\gamma =\delta :(0,0)\to (X,Y)$ is the null map. The only $0$-homotopic equivalence is $id_{(0,0)}$ and the only $0$-invariant homology is the trivial homology, we have $\cA_\partial (\cC)\cong\cA_\partial^{0-\rm hi}(\cC)$ and $\cA_\partial^0(\cC)$ is the zero category, \ie the point category as an abelian category. 
\end{remark}
\subsection{Ordinary homology} 
We now consider a category $\cC$ along with a distinguished subcategory such that $\cC$ has an initial object $0$ and a set of point objects $*$ (see \cite[Axiom 2.1.6]{UCTI}). Suppose that $(*,0)\df 0\to *$ is distinguished. We suppose that $\cC^\square$ is equipped with a collection $\Xi$ of excisive squares, see Definition \ref{excisive} and a cylinder functor $I$ on 
$\cD^\square$ for some fixed embedding $\cC^\square\into \cD^\square$, \eg $\cC=\cD$ and the cylinder functor is induced by an interval object $I\in \cC$, see Definition \ref{interval}. Denote $(\cC^\square,  \Xi, I)$ this set of data and conditions. 
\begin{defn}\label{ES}
Say that $(\cC^\square,  \Xi, I)$ is an \emph{homological structure} on $\cC$ or that $\cC$ is a homological category for short. 
Say that a relative homology on a homological category is \emph{ordinary} if satisfies the excision Axiom \ref{excisax}, the homotopy $I$-invariance Axiom \ref{hinvax} and the point axiom \cite[Axiom 2.3.2]{UCTI}. 
Say that a relative homology is \emph{extraordinary} if satisfies the previous axioms but the point axiom.

An  \emph{additive ordinary/extraordinary} homology on $\cC$ with (small) coproducts is an ordinary or extraordinary homology with values in a Grothendieck category satisfiyng  additivity \cite[Axiom 2.3.8]{UCTI}.

Denote 
$\Hom_{\rm extra} (\cC^{\square}, - )$ and $\Hom_{\rm ord} (\cC^{\square})$ (resp.\/ 
$\Hom_{\rm Extra} (\cC^{\square}, - )$ and $\Hom_{\rm Ord} (\cC^{\square}, -)$) the corresponding sub-functors of $\Hom (\cC^{\square}, -)$ (resp.\/  $\Hom_{\rm Add} (\cC^{\square}, - )$, \cf \cite[Prop.\,2.3.9]{UCTI}). 
\end{defn}
Ordinary homology corresponds to Eilenberg-Steenrod axiomatization \cite{ESF} and extraordinary to Whitehead \cite{W} generalized homologies; additivity corresponds to the wedge axiom \cite{M}. We stress out that an  homological category is somewhat a reformulation of an h-category in \cite[Chap.\,IV, Def.\,9.1]{ESF}. There is a corresponding stronger notion of site with interval introduced by Voevodsky \cite{Vh}: every homological structure $(\cC^\square,  \Xi, I)$ on a category $\cC$ with finite limits and an interval $I$ in the sense of Voevodsky yields a site with an interval by \cite{Vcd}.

Normally one would like to restrict to homologies concentrated in non-negative degrees: this is clearly possible but unnecessary, \cf \cite[Remark 2.2.5]{UCTI}. 

Note that by Proposition \ref{MV} we have Mayer-Vietoris sequences for excisive squares and by Proposition \ref{exprod} we get additivity if the coproduct exists and it is excisive.  If $I\in \cC$ is an interval in the sense of Voevodsky, by Lemmas \ref{homotopy}-\ref{hinv=hi}, we have that $I$-invariance, see Axiom \ref{Invax}, is equivalent to $I$-homotopy invariance, see Axiom \ref{hinvax}. However, one could remove or consider decorated forms of homotopy invariance. We may christen non homotopy invariant homologies the so obtained homologies without any reasonable form of homotopy invariance but the trivial one $I=1$.

\begin{thm}[Universal ordinary and extraordinary homology]\label{univES}
Let $(\cC^\square, \Xi, I)$ be an homological category. The 2-functors $\Hom_{\rm extra} (\cC^\square, - )$ and $\Hom_{\rm ord} (\cC^\square, - )$ are representable, \ie the universal extraordinary homology $H  :\cC^\square\to \cA^{\rm extra}_\partial(\cC)$  exists, is the image of the universal relative homology under the projection along a quotient $\cA_\partial(\cC)\onto \cA^{\rm extra}_\partial(\cC)$ and the universal ordinary homology $H :\cC^\square\to \cA^{\rm ord}_\partial(\cC)$ exists, being induced by a further quotient $\cA^{\rm extra}_\partial(\cC)\onto  \cA^{\rm ord}_\partial(\cC)$. 

Furthermore, for $\cC$ with (small) coproducts, $\Hom_{\rm Extra} (\cC^{\square}, - )$ and $\Hom_{\rm Ord} (\cC^{\square}, -)$ are representable by $H :\cC^\square\to \cA^{\rm Extra}_\partial(\cC)$ and $H :\cC^\square\to \cA^{\rm Ord}_\partial(\cC)$ where $\cA^{\rm Extra}_\partial(\cC)$ and  $\cA^{\rm Ord}_\partial(\cC)$ are Grothendieck quotients of $\Ind \cA^{\rm extra}_\partial(\cC)$ and $\Ind \cA^{\rm ord}_\partial(\cC)$, respectively.
\end{thm}
\begin{proof} Let $H\in \Hom (\cC^\square, \cA_\partial^{I-\rm hi}(\cC))$ be the universal $I$-homotopy invariant relative homology of Proposition \ref{hominv}, induced by 
\cite[Thm.\,2.2.4]{UCTI}, \ie the image of the universal relative homology along the quotient $\cA_\partial (\cC)\onto \cA_\partial^{I-\rm hi}(\cC)$.
Define  $\cA^{\rm extra}_\partial(\cC)$ to be the further quotient of $ \cA_\partial^{I-\rm hi}(\cC)$ by the thick subcategory generated by the objects $\ker \epsilon_i$ and $\coker \epsilon_i$ for  $\epsilon_i : H_i (Y,Z)\to H_i  (W, X)$ induced by $\epsilon : (Y, Z)\to (W, X)$ all excisive squares in $\Xi$  and $i \in \Z$.

 Adapting the arguments in the proof of \cite[Thm.\,2.2.4]{UCTI} we see that for an extraordinary homology $K\in \Hom_{\rm extra} (\cC^\square, \cA)$ the induced exact functor $r_K: \cA_\partial(\cC)\to \cA$ factors through $\cA^{\rm extra}_\partial(\cC)$ yielding $\Hom_{\rm extra}(\cC^\square,\cA)\cong \Ex_R ( \cA_\partial^{\rm extra}(\cC), \cA)$. Now define $\cA^{\rm ord}_\partial(\cC)$ to be the quotient of $ \cA_\partial^{\rm extra}(\cC)$ by the thick subcategory generated by $H_i(*,0 )\in  \cA_\partial^{\rm extra}(\cC)$ for $i\neq 0$ and get $\Hom_{\rm ord}(\cC^\square,\cA)\cong \Ex_R ( \cA_\partial^{\rm ord}(\cC), \cA)$ as claimed. 
 
For $\cC$ with arbitrary coproducts, a Grothendieck category $\cA$ and $K\in \Hom_{\rm Ord} (\cC^\square, \cA)$ we obtain that $r_K$ lifts to a unique $\Ind \cA^{\rm ord}_\partial(\cC)\to \cA$, such that  the image $H :\cC^\square\to  \Ind \cA^{\rm ord}_\partial(\cC)$ of the universal ordinary homology under $\cA^{\rm ord}_\partial(\cC)\into\Ind \cA^{\rm ord}_\partial(\cC)$ is universal with respect to ordinary homologies in Grothendieck categories, see \cite[Cor.\,2.2.6]{UCTI}. 
Define $\cA^{\rm Ord}_\partial(\cC)$ as the Grothendieck quotient of $\Ind \cA^{\rm ord}_\partial(\cC)$ by the thick and localizing subcategory generated by the kernels and cokernels of $$``\bigoplus_{k\in I}{ } " H_i (X_k)\to H_i (\coprod_{k\in I} X_k)$$ 
and, by the same argument in the proof of \cite[Prop.\,2.3.9]{UCTI}, we see that the functor $\Ind \cA^{\rm ord}_\partial(\cC)\to \cA$ factors through $\cA^{\rm Ord}_\partial(\cC)$. The category $\cA^{\rm Extra}_\partial(\cC)$ can be obtained similarly by removing the point axiom. 
\end{proof}
The point category $\cC ={\bf 1}$ is the minimal non-empty category with a unique homological structure for which we have that $\cA_\partial({\bf 1})=\cA^{\rm extra}_\partial({\bf 1})= \cA^{\rm ord}_\partial({\bf 1})$ are all zero categories. 
\begin{example}[Additive and abelian categories]  For a category $\cC$ with a zero point object $1=0$ the point axiom is trivial and $\cA^{\rm extra}_\partial(\cC)= \cA^{\rm ord}_\partial(\cC)$. If $\cC=\cR$ is additive we have the homological additive structure $(\cR^\square,  \oplus , 1)$ where we may take monos as distinguished, the biproduct excisive squares and $I=1$. Then an ordinary homology on $(\cR^\square,  \oplus , 1)$ is a generalisation of a Grothendieck homological functor on the additive category $\cR$ (see \cite[\S 3]{UCTI}); if $\cC=\cA$ is abelian then the homological structure  $(\cA^\square,  \Xi_\partial , 1)$ where $\Xi_\partial$ are the excisive squares of Example \ref{exalb} yields exactly Grothendieck homological functors, \ie $\partial$-homologies, as ordinary homologies on  
$(\cA^\square,  \Xi_\partial , 1)$ and  $\cA^\partial\cong  \cA^{\rm ord}_\partial (\cA)$ of Theorem \ref{univES} and \cite[Thm.\,3.1.3]{UCTI} coincide.  Note that there is also the abelian category of cyclic objects in an abelian category that should be properly investigated. 
\end{example}
\begin{example}[Initial and final homological category]\label{pointord} 
Any category $\cC$ with a strict initial object $0$ and a final point object $1$ is a homological category in at least two opposite ways: with the initial (or trivial or coarse) homological structure and/or with the final (or finest or discrete) homological structure. The initial structure is given by taking the image of the canonical functor ${\bf 2}\to \cC$ (it is sending $0$ to $0$ and $1$ to $1$) to be the distinguished subcategory of $\cC$ so that $\cC^\square = {\bf 2}^\square$, together with the trivial interval $I=1$ and the initial excisive squares $\Xi= {\rm Iso}$. For the initial homological structure $({\bf 2}^\square, {\rm Iso}, 1)$ we have that ordinary homologies are relative homologies on ${\bf 2}^\square$ satisfying the point axiom: as explained in \cite[Example 2.3.4]{UCTI}, we have $\cA^{\rm ord}_\partial(\cC)\cong \iRmod$. If we allow a larger distinguished subcategory, for $(\cC^\square, {\rm Iso}, 1)$ we get $\cA^{\rm ord}_\partial(\cC)\cong \cA^{\rm point}_\partial(\cC)$ and $\cA^{\rm extra}_\partial(\cC)\cong  \cA_\partial(\cC)$. 
 
Moreover, we always can get an homological structure on any $\cC$ such that $\cA^{\rm ord}_\partial(\cC)$ is the zero category, \eg by Remarks \ref{exctriv} and \ref{invtriv}. The final structure is given by taking all $\cC$ to be a distinguished subcategory so that $\cC^\square = \cC^{\bf 2}$, the final excisive squares $\Xi= \cC^{\bf 2}$ and any interval, \eg $I=0$ or $I=1$. Whence $\cA^{\rm ord}_\partial(\cC)=\cA^{\rm extra}_\partial(\cC)$ is zero, even if we allow a smaller distinguished subcategory and any cylinder functor, \ie for $(\cC^\square, \cC^\square, I)$.
\end{example} 
\begin{defn}\label{homdef}
Let $(\cC^\square, \Xi, I)$ and $(\cD^\square, \Sigma, J)$ be two homological categories. An \emph{homological functor} $F: (\cC^\square, \Xi, I)\to (\cD^\square, \Sigma, J)$ is a functor $F: \cC\to \cD$ inducing a functor $F^\square: \cC^\square\to \cD^\square$ which is compatible with generalised excisive squares and homotopies, see Definitions \ref{Ihomotopic} - \ref{univgenexc}, in the following sense: 
\begin{itemize}
\item[{\it i)}] $F^\square(\eta)$ is a $(\Sigma,J)$-generalised excision if $\eta$ is a $(\Xi ,I)$-generalised excision, 
\item[{\it ii)}] $F^\square(\gamma ) \approx_J F^\square(\delta)$ are $J$-homotopic morphims of $\cD^\square$ if $\gamma \approx_I \delta$ are two $I$-homotopic morphims of $\cC^\square$ and, additionally, 
\item[{\it iii)}]  the homological functor is \emph{ordinary} if it is also compatible with with point objects, \ie $F^\square(*,0)=(*,0)$.
\end{itemize}
 
An homological functor between categories with arbitrary coproducts is \emph{additive} if it preserves coproducts. 
\end{defn}
An ordinary homological functor is the analogue of an h-functor in 
\cite[Chap.\,IV, Def.\,9.4]{ESF}.
An ordinary homological functor induces a functor, \cf \cite[Chap.\,IV, Thm.\,9.5]{ESF}
$$H\leadsto HF^\square: \Hom_{\rm ord}(\cD^\square,\cA)\to \Hom_{\rm ord}(\cC^\square,\cA)$$
by restriction along $F^\square$ and this is natural as $\cA$ varies.

Consider $HF^\square\in  \Hom_{\rm ord}(\cC^\square,\cA^{\rm ord}_\partial(\cD))$ given by $H\in \Hom_{\rm ord}(\cD^\square,\cA^{\rm ord}_\partial(\cD))$ the universal homology on $\cD^\square$. By the universality of  $\cA^{\rm ord}_\partial(\cC)$ the latter homology $HF^\square$ yields an $R$-linear exact functor $$F_\partial \df r_{HF^\square}: \cA^{\rm ord}_\partial(\cC)\to \cA^{\rm ord}_\partial(\cD)$$ and, similarly, we get $F_\partial: \cA^{\rm Ord}_\partial(\cC)\to \cA^{\rm Ord}_\partial(\cD)$ in the additive case or $F_\partial: \cA^{\rm extra}_\partial(\cC)\to \cA^{\rm extra}_\partial(\cD)$ for any homological functor.
\begin{defn}\label{homeqdef}
An homological functor $F: (\cC^\square, \Xi, I)\to (\cD^\square, \Sigma, J)$ is an \emph{homological equivalence} if $F_\partial: \cA^{\rm extra}_\partial(\cC)\longby{\simeq} \cA^{\rm extra}_\partial(\cD)$ is an equivalence; $F$ ordinary (resp.\/ additive) is an homological equivalence if $F_\partial: \cA^{\rm ord}_\partial(\cC)\longby{\simeq} \cA^{\rm ord}_\partial(\cD)$ (resp.\/ $F_\partial: \cA^{\rm Ord}_\partial(\cC)\longby{\simeq} \cA^{\rm Ord}_\partial(\cD)$) is an equivalence. 
\end{defn}

Clearly, there is a canonical sufficient condition to obtain compatibility with homotopies for a functor between categories endowed with cylinder functors: that $F^\square$ is commuting with the cylinder functors, \ie $F^\square I= JF^\square$ and $F^\square(p)=p_{F^\square}$. However, this stronger condition is not necessary. 
In fact, for any homological category $(\cC^\square, \Xi, I)$ the identity functor induces an homological functor from the initial homological structure $(\cC^\square, {\rm Iso}, 1)\to (\cC^\square, \Xi, I)$ 
which is not, in general, commuting with the cylinder functor. Moreover, the identity functor also induces an ordinary homological functor $(\cC^\square, \Xi, I)\to (\cC^\square, \cC^\square, I)$ to the final homological structure, \cf Example \ref{pointord}. These functors induce the projection $\cA^{\rm point}_\partial(\cC)\onto \cA^{\rm ord}_\partial(\cC)$ and the zero functor, respectively.
For $\cC$ with a strict initial object $0$ and a final point object $1$ we also have a unique ordinary homological functor 
 $$!: ({\bf 2}^\square, {\rm Iso}, 1)\to (\cC^\square, \Xi, I)$$
and the trivial functor $\text{!`} : (\cC^\square, \Xi, I)\to ({\bf 1}, {\bf 1}, 0)$. The homological functor $!$ is then inducing the functor  $$!_\partial :  \iRmod\to \cA^{\rm ord}_\partial(\cC) \ \ \ |R|\leadsto H_0(1) $$
which is also the composition of the inclusion $!_{\rm point}: \iRmod\into \cA^{\rm point}_\partial(\cC)$ and the projection $\cA^{\rm point}_\partial(\cC)\onto \cA^{\rm ord}_\partial(\cC)$, compatibly with \cite[Prop.\,2.3.6]{UCTI}. This functor uniquely extends
$$!_\partial :  \iRMod\to \cA^{\rm Ord}_\partial(\cC)$$
along the canonical exact $R$-linear functor $\cA^{\rm ord}_\partial(\cC) \to \cA^{\rm Ord}_\partial(\cC)$ such that $|R|\leadsto H_0(1)$. 
\begin{defn}\label{coeff} 
Let $(\cC^\square, \Xi, I)$ be a homological category  with a strict initial object $0$ and a final point object $1$.
For $H$ any ordinary homology we say that $H_0(1)$ is the \emph{coefficient object} of the homology; if $H$ is extraordinary $H_i(1)$ are the coefficients. For a distinguished global point $1\to X$ the \emph{reduced} homology is 
$$\tilde{H}_i(X)\df H_i(X, 1)\cong H_i(X)/H_i(1)\cong \ker H_i(X)\to H_i(1)$$ 

Let $H :\cC^\square\to \cA^{\rm ord}_\partial(\cC)$ (resp.\/ $H :\cC^\square\to \cA^{\rm Ord}_\partial(\cC)$) be the universal (resp.\/ additive) ordinary homology. 
Denote $H_0(1) \df \underline{R}\in \cA^{\rm ord}_\partial(\cC)$ (resp.\/ $\underline{R}\in\cA^{\rm Ord}_\partial(\cC)$) such that $!_\partial (|R|)= \underline{R}$ for $|R|\in \iRmod$ the universal hieratic $R$-module (see \cite[Def.\,1.3.3]{UCTI}). Say that $\underline{R}$ is the \emph{universal coefficient object} of the ordinary homology theory with respect to the homological structure: if $\underline{R}= 0$ we say that the homological structure is trivial.
\end{defn}
For  any relative homology $H$ and $1\to Y\to X$ distinguished we obtain a long exact sequence
$$\cdots \to \tilde{H}_i(Y)\to \tilde{H}_i(X)\to H_i(X, Y)\to \tilde{H}_{i}(Y)\to \cdots$$
and if $H$ is ordinary we have $H_i(X)\cong  \tilde{H}_i(X)$ for $i\neq 0$ and  $H_0(X)\cong  \tilde{H}_0(X)\oplus H_0(1)$. 
\begin{propose}
If there is an ordinary homology $H'  :\cC^\square\to \cA$ with a non trivial coefficient object $H_0'(1)\neq 0$ then the homological structure is not trivial. 
\end{propose}
\begin{proof}
In fact, consider $r_{H'}: \cA^{\rm ord}_\partial(\cC) \to \cA $ with $H_0'(1) \df A \neq 0$. The universal object $\underline{R}\in \cA^{\rm ord}_\partial(\cC)$ is such that $r_{H'}(\underline{R})= A$ and therefore $\underline{R}\neq 0$ is not trivial. 
\end{proof}
If $H'$ is an ordinary homology $H': \cC^\square\to \cA$ such that $H_0'(1) \df A \in \cA$, $A\neq 0$,  then $r_A= r_{H'} !_\partial : \iRmod\to \cA$ is not the zero functor. 
The projection $$\cA^{\rm ord}_\partial(\cC)\onto \cA (H')\df \cA^{\rm ord}_\partial(\cC)/\ker r_{H'}$$ and the functors $!_\partial $ and $r_{H'}$ are then inducing a factorisation 
$$\xymatrix{&\cA^{\rm ord}_\partial(\cC)\ar[d]^{}\ar@/^1.2pc/[ddr]^{r_{H'}}&\\
&\cA (H')\ar[dr]^{\bar{r}_{H'}}&\\  
\iRmod\ar@/^1.2pc/[uur]^{!_\partial}\ar[ur]^{ }\ar[rr]_{r_A}& &\cA}$$
such that $\cA (H')\neq 0$, $\bar{r}_{H'}$ is faithful and essentially surjective if $r_A$ is a quotient. For $\cA$ Grothendieck and $H'$ additive we obtain a similar diagram where 
$\cA^{\rm ord}_\partial(\cC)$ is replaced by $\cA^{\rm Ord}_\partial(\cC)$ and $\iRmod$ is replaced by its indization $\iRMod$. 
\begin{remark}[Transfers]\label{trans}
 We clearly can add ``finite coverings'' to the data of an homological structure defining an ordinary homology with transfers following \cite[Remark 2.2.10]{UCTI}. The universal ordinary homology theory with transfers $H^\tr: \cC^\square \to \cA^{\rm ord/\tr}_\partial(\cC)$ exists and it is obtained by representing the functor $\Hom_{\rm ord/\tr} (\cC^\square, - )$ given by ordinary homologies with transfers. In fact, this theory is constructed by considering the universal relative homology with transfers $H^\tr : \cC^\square \to \cA_\partial^\tr(\cC)$ and arguing as in the proof of Theorem \ref{univES}. There is a canonical $R$-linear exact functor $$r_{H^\tr}: \cA^{\rm ord}_\partial(\cC)\to \cA^{\rm ord/\tr}_\partial(\cC)$$ Similarly, arguing with $H^\tr  : \cC^\square \to \Ind \cA^{\rm ord/\tr}_\partial(\cC)$ we obtain $H^\tr  : \cC^\square \to \cA^{\rm Ord/\tr}_\partial(\cC)$ the universal additive ordinary homology with transfers. 
\end{remark}

\section{Universal singular homology}
\subsection{Singular homology}
Let $\SSet\df \Psh (\Delta)$ be the category of simplicial sets. For any Grothendieck category $\cA$ together with a fixed object $A\in\cA$ we obtain a canonical functor
$$C(-, A): \SSet \to \Ch (\cA)$$
whose essential image is contained in the subcategory of ordinary (\ie concentrated in non negative degrees) chain complexes.  The chain complex $C(X_\d, A)$ associated with a simplicial set $X_\d$ is simply the chain complex induced by the following simplicial object 
$$m\to n\in \Delta\leadsto X_n \longby{\varphi} X_m \leadsto \bigoplus_{x\in X_n}{}  A_x \longby{\varphi_A} \bigoplus_{y\in X_m}{}  A_y \in \cA$$ where the direct sums are taken over copies of $A$ and $\varphi_A$ is induced via $\varphi$ by universality of the direct sum. Note that for finite simplicial sets the direct sums are over finite sets and we may avoid the assumption on the abelian category $\cA$ to be Grothendieck. The association 
$X_\d\leadsto C(X_\d, A)$ is clearly functorial. 
Any cosimplicial object $\Im$ in a (locally small) category $\cC$ yields 
$$X\leadsto {\rm Sing}^\Im_\d (X)\df \cC (\Im^\d, X): \cC\to \SSet$$
a singular simplicial set functor: we shall assume that $\Im^0\cong 1$ is a final point object and $\Im^1\cong I$ is a non-trivial interval object. 
For $\cC$ with arbitrary colimits, universality of Yoneda embedding of $\Delta$ in $\SSet$ implies that ${\rm Sing}^\Im_\d $ has a left adjoint, the so called ``geometric realisation'' functor as shown
$$\xymatrix{\Delta\ar[r]^{} \ar[d]_{\Im} & \Psh (\Delta) = \SSet \ar@/^0.8pc/[dl]^{{\rm Re}^\Im} \\ 
\ar@/^0.8pc/[ur]_{{\rm Sing}^\Im}\cC}$$
where ${\rm Re}^\Im (X_\d)$ is the colimit of $\Im^n$ over $\Delta^n\to X_\d$ \ie in such a way that we moreover have $\cC({\rm Re}^\Im X_\d, X)\cong \SSet (X_\d , {\rm Sing}^\Im_\d (X))$ by construction. Setting 
$${\rm Sing}^\Im_A (X)\df C({\rm Sing}^\Im_\d (X), A)$$ and ${\rm Sing}^\Im_A (X, Y)\df {\rm Cone}\ {\rm Sing}^\Im_A (Y)\to  {\rm Sing}^\Im_A (X)$ for pairs $(X, Y)\in \cC^\square$ provide,  by the well known properties of the cone, a well defined functor
${\rm Sing}^\Im_A: \cC^\square\to \Ch (\cA)$
such that
$$(X,Y)\leadsto H^{\rm Sing}_i(X, Y; A)\df H_i({\rm Sing}^\Im_A (X, Y)):\cC^\square\to \cA$$ 
is a relative homology $H^{\rm Sing}(-;A): \cC^\square\to \cA$  which is concentrated in non negative degrees. 
\begin{lemma} \label{relsing}
Let $\cC$ be a category endowed with a distinguished subcategory, a strictly initial object $0$ and a final point object $1$, a cosimplicial object $\Im$ (such that $\Im^0\cong 1$ and $\Im^1\cong I$) and $(1,0)\in \cC^\square$.  For $A\in \cA$  Grothendieck and $H^{\rm Sing}(-;A): \cC^\square\to \cA$ we have that $H^{\rm Sing}_0(1;A)=A$ is the coefficient object and $H^{\rm Sing}_i(1;A)=0$ for $i\neq 0$. 
In particular, for $\cA =  \iRMod$ and $A=|R|$ we obtain $$H^{\rm Sing}\df H^{\rm Sing}(-;|R|): \cC^\square\to \iRMod$$ representing the subfunctor $\Hom_{\rm Sing}(\cC^\square, -)\subset \Hom (\cC^\square, -)$ of singular relative homologies. If $\cC$ has (small) coproducts and $\Im$ is connected, \ie ${\rm Sing}^\Im_\d (X) = \coprod_k {\rm Sing}^\Im_\d (X_k)$ for $X = \coprod_k X_k$, then $\Hom_{\rm Sing}(\cC^\square, -)\subset \Hom_{\rm Add} (\cC^\square, -)$ and $H^{\rm Sing}$ is additive. 
\end{lemma}
\begin{proof}
 Since $0$ is strictly initial ${\rm Sing}^\Im_A (0)=0$. Therefore ${\rm Sing}^\Im_A (X, 0) ={\rm Sing}^\Im_A (X)$. For $X=1$ we have that ${\rm Sing}^\Im_\d (1) = \Delta^0$ and $C({\rm Sing}^\Im_\d (1), A)$ is the complex with $A$ in all degrees; since the differential is zero in odd degrees and the identity in even degrees $H^{\rm Sing}(-;A)$ satisfies the point axiom.

For $\cA = \iRMod = \Ind \iRmod$, $A=|R|$ and ${\rm Sing}^\Im_{|R|} (X, Y)$ as above, we have that
$$(X,Y) \leadsto {\rm Sing}^\Im_{|R|} (X, Y): \cC^\square \to \Ch (\iRMod)$$
is a canonical functor and $H^{\rm Sing}_i(X, Y)\df H_i({\rm Sing}^\Im_{|R|} (X, Y))\in \iRMod$ is a relative homology.

Moreover,  for any $R$-linear Grothendieck category $\cA$ we have that $A\in\cA$ induces a unique $R$-linear exact functor $r_A: \iRmod\to \cA$ such that $r_A(|R|)=A$ and $r_A$ uniquely extends to an $R$-linear exact functor which is preserving filtered colimits $$R_A\df \Ind r_A \circ L: \iRMod\to \cA$$ 
where $\Ind r_A : \Ind\iRmod\to \Ind\cA$ and $L: \Ind\cA\to \cA$  (see \cite[Prop.\,1.3.6]{UCTI}) such that $R_A  (H^{\rm Sing})= H^{\rm Sing}(-;A)$. In fact, we also get the functor 
$$\Ch R_A: \Ch (\iRMod) \to \Ch (\cA)$$
such  that $C(-, A)$ factors through $\Ch R_A$ and $C(-, |R|)$.

Finally, if the cosimplicial object $\Im$ is connected then ${\rm Sing}^\Im_A (X)= \oplus_k {\rm Sing}^\Im_A (X_k)$ for $X = \coprod_k X_k$ thus $H^{\rm Sing}(-;A)$ is additive. 
\end{proof}
Note that associating $H^{\rm Sing}(-;A)\leadsto H^{\rm Sing}_0(1;A)=A$ yields $\Hom_{\rm Sing}(\cC^\square, \cA)\by{\simeq} \cA$ an equivalence for $\cA$ Grothendieck. Actually, $\iRMod$ is representing this latter forgetful functor on Grothendieck categories.
\begin{thm}[Universal singular homology]\label{unising}
If $(\cC^\square, \Xi, I)$ is an homological structure on a category $\cC$ with (small) coproducts endowed with a connected cosimplicial object $\Im$ (such that $\Im^0\cong 1$ and 
$\Im^1\cong I$) then there is an additive ordinary homology $$H^{\rm Sing}=\{H^{\rm Sing}_i\}_{i\in \N}: \cC^\square\to \cA^{\rm Sing}_\partial(\cC)$$ which is representing $\Hom_{\rm Sing}(\cC^\square, -)\subset \Hom_{\rm Ord} (\cC^\square, -)$, \ie $H^{\rm Sing}$ is universal with respect to (additive) ordinary singular homologies.
\end{thm}
\begin{proof} Let $H^{\rm Sing}(-;A): \cC^\square \to \cA$ be a singular homology which is ordinary with respect to the homological structure.  
From the proof of Lemma \ref{relsing} we obtain the exact functor $R_A : \iRMod\to \cA$
such that  $R_A  (H^{\rm Sing})= H^{\rm Sing}(-;A)$. Let $\cA^{\rm Sing}_\partial(\cC)$ be the Grothendieck quotient of $\iRMod$ by the thick and localizing subcategory generated by $\im (\gamma_i-  \delta_i)$  for $\gamma_i, \delta_i: H_i^{\rm Sing} (X, Y)\to H_i^{\rm Sing}  (X', Y')$  all parallel $I$-homotopic morphisms  $\gamma \approx_I \delta: (X,Y)\to (X',Y')$ and by $\ker \epsilon_i$ and $\coker \epsilon_i$ for  $\epsilon_i : H_i^{\rm Sing} (Y,Z)\to H_i^{\rm Sing}  (W, X)$ induced by $\epsilon : (Y, Z)\to (W, X)$ all excisive squares in $\Xi$. The image $H^{\rm Sing}: \cC^\square \to \cA^{\rm Sing}_\partial(\cC)$ of $H^{\rm Sing}$ under the projection $$\pi: \iRMod\onto\cA^{\rm Sing}_\partial(\cC)$$ is the universal additive ordinary singular homology with coefficients in $\pi(|R|)$. In fact, the functor $R_A$ factors through $\cA^{\rm Sing}_\partial(\cC)$ because, by assumption, $H^{\rm Sing}(-;A)$ is ordinary. 
\end{proof}
We always have an exact functor (which is not a quotient, in general) $$r_{H^{\rm Sing}}: \cA^{\rm Ord}_\partial(\cC)\to \cA^{\rm Sing}_\partial(\cC)$$ but the existence of a non trivial ordinary singular homology is not granted and it is equivalent to the fact that the category $\cA^{\rm Sing}_\partial(\cC)$ is not trivial. Moreover, $\ker \pi\subset \iRMod$ is localizing, by the construction in the proof of Theorem \ref{unising}, thus  $\pi$ has a section $\cA^{\rm Sing}_\partial(\cC)\into \iRMod$. 
\begin{cor} \label{ord=sing}
Same assumptions and notation as in Theorem \ref{unising}.
If the universal ordinary homology $H$ coincides with the singular homology with universal coefficients $H = H^{\rm Sing}(- ;\underline{R}): \cC^\square \to \cA^{\rm Ord}_\partial(\cC)$ then $\cA^{\rm Ord}_\partial(\cC)\by{\simeq} \cA^{\rm Sing}_\partial(\cC)$ and $H= H^{\rm Sing}$. Moreover, all additive ordinary homologies in Grothendieck categories are singular homologies and conversely.
\end{cor}

\subsection{Homotopy versus homology} 
Given a (small) category $\cC$,  Dugger \cite{DU} has shown that there is a universal model category $U\cC$ built from $\cC$. The universal way of expanding $\cC$ into a model category is nicely and simply given by $$U\cC\df \SPsh (\cC)$$ the category of simplicial presheaves endowed with the Bousfield-Kan model 
structure (see \cite[Prop.\,2.3]{DU} for details on the universal property).  The model category $U\cC$ is providing the universal homotopy theory on $\cC$. Actually, this $U\cC$ is somewhat the free homotopy theory on $\cC$ and the standard technique of localization provides a method for imposing ``relations'' on $U\cC$.  

For any Grothendieck category $\cA$, consider a model structure on the category of chain complexes $\Ch (\cA )$ (see \cite{CD}).
Therefore, for any category $\cC$, we also obtain a model category
 $$M\cC\df \Ch (\Ind \cA (\cC))$$ whose underlying category is the category of chain complexes in the Grothendieck category $\Ind\cA (\cC)$ determined by  the universal homology $H$ (see \cite[Thm.\,2.1.2]{UCTI}): this yields a canonical functor from $\cC$ to $\Ch (\Ind \cA (\cC))$ sending $$X\leadsto M(X)\df ``\bigoplus_{i\in\Z}" H_i(X)[i]$$ the ``tautological" complex having $H_i(X)$ in degree $i$ and with zero differentials. Therefore we obtain a Quillen pair 
$$U\cC\qto M\cC$$
which is an abstract version of the Hurewicz map. Actually, we can get several versions of this pair, \eg for any decoration $\cA^\dag (\cC)$. 
\begin{example}
For $\cC = {\bf 1}$, the point category, the universal homotopy and the universal homology theories yield a Quillen pair
$$\SSet \qto \Ch (\iRMod)$$
which is universal among all Quillen pairs between $\SSet$ and $\Ch (\cA)$ induced by an object of a Grothendieck category $\cA$. In fact, given $A\in \cA$, \ie a functor ${\bf 1}\to \cA$,  we obtain the exact functor $\Ch R_A :\Ch (\iRMod)\to \Ch (\cA)$ such that $R_A(|R|) = A$ as $|R|\in \iRMod$ is the universal object.  
\end{example}
Similarly, the universal relative homology $H$ (see \cite[Thm.\,2.2.4]{UCTI}) yields a model category 
$$M\cC^\square\df \Ch (\Ind \cA_\partial(\cC))$$ 
along with the canonical functor $M : \cC^\square \to M\cC^\square$ given by the relative ``tautological" complex. Therefore 
$$U\cC^\square \qto M\cC^\square$$
there is a canonical Quillen pair.  Now any decoration $\cA_\partial^\dag(\cC)$ induces a corresponding decorated Quillen pair. For an homological category $(\cC^\square, \Xi, I)$ where $I$ is an interval in the sense of Voevodsky let $(U\cC^\square)_I$ be the Bousfield localisation of $U\cC^\square$ at the class $\{p: I(X, Y)\to (X,Y)\}$ indexed by $(X,Y)\in \cC^\square$ (see \cite[Def.\,5.4]{DU}). Denote by $$M^{\rm Ord}\cC^\square\df \Ch (\cA_\partial^{\rm Ord}(\cC))$$ the model category of chain complexes in the Grothendieck category $\cA_\partial^{\rm Ord}(\cC)$.
\begin{thm}[Homotopy versus homology]\label{hothom}
For any interval object $I$ in the sense of Voevodsky and an homological structure $(\cC^\square, \Xi, I)$ on a category $\cC$ there is a  canonical functor $M^{\rm Ord} : \cC^\square \to M^{\rm Ord}\cC^\square$ inducing 
$$(U\cC^\square)_I \qto M^{\rm Ord}\cC^\square$$
a Quillen pair which is universal with respect to Quillen pairs induced by additive ordinary homologies. 
\end{thm}
\begin{proof}
The functor $M^{\rm Ord}$ is obtained by composition of $(X,Y)\leadsto \oplus H_i(X,Y)[i]$   with the canonical functor $\Ch(\Ind \cA_\partial(\cC))\to \Ch (\cA_\partial^{\rm Ord}(\cC))$ which is induced by the exact functor from $\cA_\partial(\cC)$ to $\cA_\partial^{\rm Ord}(\cC)$ given by the proof of Theorem \ref{univES}. The Quillen pair is induced by the universal property of $U\cC^\square$ and localisation (see \cite[Prop.\, 2.3]{DU} and Lemma \ref{hinv=hi}). The claimed universality is induced by that of $\cA_\partial^{\rm Ord}(\cC)$ in Theorem \ref{univES}.  
\end{proof}
Note that for $\cC$ with a Voevodsky interval object $I$ and a Grothendieck topology $\tau$ the localised category $(U \cC/\tau)_{I}$ is playing an important r\^ole in the construction of several motivic homotopy theories (here $U\cC/\tau$ is the localised model category considered in \cite[Def.\,7.2]{DU}).   The classical Morel-Voevodsky model structure on $\cC=\Sm_k$, can be recovered from $U \Sm_k$ as the localization $(U \Sm_k/\Nis)_{\Aff^1_k}$ where $I =\Aff^1_k$ is the affine line and $\tau = \Nis$ is the Nisnevich topology (see \cite[Prop.\,8.1]{DU}).

\subsection{Cellular homology}
We are now going to explore some additional properties of our category $\cC$ along with its distinguished subcategory. For simplicity, we assume that all objects of $\cC$ are distinguished and that all isomorphisms of $\cC$ are distinguished. For a fixed $X\in\cC$  consider filtrations $X_\d\to X$ as shown
$$\xymatrix{X& \ar[l]^-{}  \cdots & \ar[l]^-{}  X_{p} & \ar[l]^-{}X_{p-1}& \ar[l]^-{}\cdots &\ar[l]^-{}X_0& \ar[l]^-{}X_{-1}=0 }$$ 
for families of distinguished (mono)morphisms $X_p\to X$ and $p$ is a non-negative integer. 
The following is a reformulation of \cite[Lemma 3.1.2]{BV}.
\begin{lemma}\label{2=0} Let $H\in \Hom (\cC^\square, \cA )$  be a relative homology.
Any triple of distinguished morphisms 
$X_{p-3}\to X_{p-2}\to X_{p-1}\to X_p$, \eg the distinguished monomorphisms in a filtration $X_\d\to X$, yields
$$\xymatrix{H_i(X_p,X_{p-1})\ar[r]^{\partial_i} \ar@/^1.7pc/[rr]^{\rm zero} & H_{i-1}(X_{p-1}, X_{p-2})\ar[r]^{\partial_{i-1}} & H_{i-2}(X_{p-2},X_{p-3})}$$
in  $\cA$ for all $i\in \Z$.
\end{lemma}
\begin{proof} 
By naturality we get the following commutative triangles
$$
\xymatrix{& H_{*-1} (X_{p-1}) \ar[d]^-{} & \\
H_*(X_p,X_{p-1})\ar[ur]^-{}\ar[r]^-{\partial_*} & H_{*-1}(X_{p-1},X_{p-2})\ar[r]^-{\partial_*}\ar[dr]^-{}  &  H_{*-2}(X_{p-2},X_{p-3})\\
& & H_{*-2}(X_{p-2})\ar[u]}
$$
By exactness we then get that $\partial_*^2=0$ as claimed.
\end{proof}
\begin{defn}\label{ordcell}
Say that $(X, Y)\in \cC^\square$ is a \emph{good or cellular pair} for $H\in \Hom(\cC^\square, \cA)$ if either $Y\cong X$ or exists an integer $n$ such that $H_{i}(X, Y)=0$ for $i\neq n$ and  $H_{n}(X, Y)\neq 0$. For a filtration $X_\d\to X$ say that it is a \emph{good or cellular filtration} for $H$  if $(X_p, X_{p-1})$ is a cellular pair such that 
\begin{itemize}
\item $H_{i}(X_p, X_{p-1})=0$ for $i\neq p$,
\item $H_i(X_p)=0$ for $i>p$ and
\item $H_i(X_p)\cong H_i(X)$ for $i<p$, induced by $X_p\to X$.
\end{itemize}
Say that $H$ is \emph{locally cellular} if we have good filtrations for $H$ on a subcategory of $\cC$.
Say that the \emph{homological structure is cellular} if the universal ordinary homology $H\in \Hom(\cC^\square, \cA_\partial^{\rm Ord}(\cC))$ is locally cellular.
\end{defn}
Cellularity immediately implies that $H_i(X_0)= H_i(X)$ for $i<0$. If $X = X_d$ in a good filtration then  $H_{i}(X)=0$ for $i>d$. Note that when the homological structure is cellular  all ordinary homologies are cellular.   Moreover, if $\cA$ is an abelian category and $\cP\subseteq \cA$ a $\mathfrak{p}$-subcategory, see \cite[Def.\,1.3.14]{UCTI}, we may say that $H\in \Hom(\cC^\square, \cA)$ is $\mathfrak{p}$-cellular if the non zero homology objects in a good filtration are in $\cP$. For a tensor category $(\cA, \otimes)$ as in the assumptions of  
\cite[Def.\,1.4.1]{UCTI} we may say that $H$ is $\flat$-cellular if the non zero homology is in $\cA^\flat$.

For $H\in \Hom(\cC^\square, \cA)$ locally cellular and a good filtration $X_\d\to X$ we get the following commutative diagram with exact horizontal rows
$$\xymatrix{H_{i+1}(X_{i+1}, X_i)\ar[rd]_{\partial_{i+1}}\ar[r]^{}& H_{i}(X_i)\ar[r]^{\pi_i}\ar[d]^{\beta_i}& 
H_i(X)& \\
 &H_{i}(X_{i}, X_{i-1})\ar[rd]_{\partial_{i}}\ar[r]^{}& H_{i-1}(X_{i-1})\ar[d]^{\beta_{i-1}}\ar[r]^{\pi_{i-1}} &H_{i-1}(X)\\
 && H_{i-1}(X_{i-1}, X_{i-2})
}$$ 
and let 
$$C_*^H(X):\xymatrix{\cdots \to H_i(X_i,X_{i-1})\ar[r]^{\ \ \partial_i\ \ }  & H_{i-1}(X_{i-1}, X_{i-2})\ar[r]^{\partial_{i-1}\ \ } & H_{i-2}(X_{i-2},X_{i-3})\to\cdots}$$
be the chain complex in $\cA$ given by Lemma \ref{2=0}. The complex $C_*^H( - )$ is functorial with respect to good filtrations \ie morphisms $f: X\to X'$ inducing morphisms of pairs $(X_p, X_{p-1})\to (X_p', X_{p-1}')$ on the corresponding good filtrations. The topological version of the following is well-known, \cf \cite[Thm.\,2.35]{Ha}.
\begin{propose}[Cellular complex] \label{cellularity}
If $H\in \Hom(\cC^\square, \cA)$ is locally cellular and  $X_\d\to X$ is a good filtration then there is a canonical isomorphism in $\cA$ $$H_{i}(X) \cong H_i(C_*^H(X))$$  with the homology of the cellular complex $C_*^H(X)$ functorially with respect to morphisms of good filtrations. 
\end{propose}
\begin{proof}
The computation of the homology of $C_*^H(X)$ is straightforward: follows from the fact that in the previous diagram $\beta_i$ are monos and $\pi_i$ are epis.  
\end{proof}
Note that if $H_i(X)$ are projective objects of $\cA$ we also have 
that $H_i(X)$ is a direct summand of $H_i(X_i)$ and we get $$\epsilon: \bigoplus_{i}^{} H_i(X)[i]\longby{\qi} C_*^H(X)$$
a quasi-isomorphism. In fact, $\pi_i$ do have a section $\sigma_i:H_i(X)\to H_i(X_i)$ if $H_i(X)$ is projective and    we get a chain morphism
 $$\epsilon_i\df \beta_i\sigma_i:H_i(X)[i]\to C_i^H(X)=H_{i}(X_{i}, X_{i-1})$$  by construction. 
\begin{example}
In particular, if an ordinary theory $H'\in \Hom_{\rm ord}(\cC^\square, \cA)$ on a homological category $(\cC^\square, \Xi, I)$ is locally cellular we have that the universal ordinary homology induces the following
$$\xymatrix{\cC^\square\ar[r]^{} \ar@/^1.7pc/[rr]^{H}&\cA^{\rm ord}_\partial(\cC)
\ar@/_1.7pc/[rr]_{r_{H'}}\ar[r]^{}& \cA (H')\ar[r]^{\bar r_{H'}} & \cA}$$
where the universal theory $H$ is also locally cellular since $\bar r_{H'}$ is faithful.  For example, let $\cC =\Sch_S$ be a suitable category of schemes (of finite type) over a (Noetherian) base scheme; let $\Sch_S^\square$ be the category whose objects are $f: Y\into X$ locally closed $S$-immersions, \ie $f= j\circ i$ where $i$ is a closed immersion and $j$ is an open immersion. Any Voevodsky $cd$-structure $\Xi_{cd}$ (see \cite{Vcd} and \cite{Vucd}) yields a class of excisive squares. The standard one is given by the usual Zariski open coverings $X=U\cup V$ where $\epsilon_X : (U, U\cap V) \to (U\cup V, V)\in \Sch^\square_S$ are excisive squares and the excisive squares given by the abstract blow-ups, \ie $f: Y\to X$ proper surjective $S$-morphism such that exists $Z\into X$ closed  where $Z'=f^{-1}(Z) $ and $Y- Z'\by{\simeq} X-Z$ is an isomorphism. 
For any $cd$-structure $\Xi_{cd}$ on $\Sch_S$  and $I=\Aff_S^1$ we get an homological structure $(\Sch_S^\square, \Xi_{cd}, \Aff_S^1)$ on $S$-schemes and the universal (additive) ordinary homology $$H: \Sch_S^\square\to \cA^{\rm ord}_\partial(\Sch_S)$$ 
which is the  abelian version of Voevodsky triangulated category of (constructible, effective) motives ``without transfers". Similarly, we obtain $H^\tr: \Sch_S^\square\to \cA^{\rm ord/\tr}_\partial(\Sch_S)$ the version ``with transfers'' for finite (surjective) morphisms. For $S=\Spec (k)$ the spectrum of a subfield $k\subset \C$ of the complex numbers we have that $H' = H^{\rm Sing}(-; R)$ is locally cellular with respect to affine schemes of finite type over $k$, thanks to Nori's Basic Lemma, see \cite[Thm.\,2.5.2]{HMS} and ordinary with respect to the standard homological structure, \cf Remark \ref{toprk} c). 
\end{example}
\begin{remark}
Let $(\Sch_k, \Xi_{cd}, \Aff^1_k)$ be an homological category of $k$-algebraic schemes over a field $k$. Consider ordinary homology theories $H$ such that
\begin{itemize}
\item[{\it a)}] $H_i(X,Y)=0$ is vanishing for $X$ affine and $i> d=\dim (X)$, and \item[{\it b)}] there is a ``relative duality'' isomorphism  $$H_i(X\setminus Y, Z\setminus (Y\cap Z))\cong H_{2d-i}(X\setminus Z, Y \setminus (Y\cap Z)$$ for $X$ proper smooth, $Y$ and $Z$ normal crossings such that $Y\cup Z$ is a normal crossing.
\end{itemize}
Further, eventually, require that this isomorphism is compatible with the long exact sequence of relative homology, \eg this is the case of $H^{\rm Sing}(-; K)$ singular homology with coefficients in a field $K$ (on $\Sch_k$ for $k\subset \C$).
For any field $K$ we can get an ordinary homology  $H: \Sch_k^\square\to \cA^{\rm dual}_\partial(\Sch_k)$ in a $K$-linear abelian category where these additional axioms hold true universally along with an exact $K$-linear functor 
$$r_{H}:\cA^{\rm ord}_\partial(\Sch_k)\to \cA^{\rm dual}_\partial(\Sch_k)$$
(and this latter functor is not trivial). In fact, adding $(X\setminus Y, Z\setminus (Y\cap Z), i)\to (X\setminus Z, Y \setminus (Y\cap Z), 2d-i)$ to the Nori diagram one follows the track of the proofs of \cite[Thm.\,2.2.4]{UCTI} and Theorem \ref{univES} to construct $\cA^{\rm dual}_\partial(\Sch_k)$. For an infinite field $k$ which admits resolution of singularities the homology $H$ is locally cellular over $\Sch_k$ by exactly the same arguments in the direct proof of the Basic Lemma, see the proof of \cite[Thm.\,2.5.2]{HMS}. 
\end{remark}

\section{Topological motives}
\subsection{Convenient spaces}
Let ${\rm CW}$ be the category of CW-complexes. 
Let  ${\rm CW}\subseteq \Top \subset {\rm Top}$ be a convenient category of topological spaces (\eg see \cite{St}) containing the category  of spaces having the homotopy type of a CW-complex: we get the standard homological structure $(\Top^\square, \Xi, I)$ on $\Top$ and ${\rm CW}$ as usual. 

Assume the subspaces to be distinguished so that $\Top^\square\subset \Top^{\bf 2}$ is the usual category of topological pairs. We have $0= \emptyset$ and the point objects $1= \{*\}$ are given by any singleton space. 

\begin{defn} \label{stop}
Let $\Xi = \{\cE_X\}_{X\in \Top}$ where $\cE_X=\{\epsilon_X\}$ is the family of excisive squares  $\epsilon_X : (U, U\cap V) \to (X, V)\in \Top^\square$ where $X =  \mathring{U}\cup \mathring{V}$ is the union of the interiors.  The closed real interval $I =[0,1]$  (or $I= \R$  the affine real line) is an interval object of $\Top$. Let $(\Top^\square, \Xi, I)$ be the \emph{standard homological structure} on $\Top$. 
\end{defn}

Excision and homotopy invariance imply that also $\epsilon_{X} : (Y,Z)\to (X, 1)\in \Top^\square$ is excisive for the pointed space $Z/Z\to X=Y/Z$ together with a subspace $S\subseteq Y$ such that $\bar{Z} \subseteq \mathring{S}$ and $Z\to S$ is a deformation retract, \eg if $(Y,Z)$ is a CW pair.  

For any convenient category $\Top$  of topological spaces we may introduce the following categories: the abelian categories $\cA^{\rm Ord}_\partial(\Top)$ and $\cA^{\rm Extra}_\partial(\Top)$, \ie the ordinary and extraordinary (effective) ``topological motives"  respectively: the topological ``motivic'' homology 
$$H: \Top^\square \to \cA^{\rm Ord}_\partial(\Top)$$
 is the universal ordinary homology. 
 
\begin{remarks}\label{toprk}
a) Similarly, the abelian categories $\cA^{\rm ord}_\partial(\Top)$ and $\cA^{\rm extra}_\partial(\Top)$ are categories of ``constructible'' topological motives and there are versions of topological motives ``with transfers'' by Remark \ref{trans}. There is an induced homological structure on the category of ``spaces with coefficients'' $\Top_\nabla$ given by $(X, E)$ a bundle $E$ of abelian groups on $X$ and bundle maps
$(X, E)\to (X', E')$ providing $\cA^{\rm Ord}_\partial(\Top_\nabla)$.\\

b) Note that the category ${\rm FTop}\subset \Top$ of finite topological spaces is also endowed with a standard homological structure $({\rm FTop}^\square, \Xi, I)$ even if $I\notin {\rm FTop}$. Thus we also get finite topological motives and a motivic homology  $H: {\rm FTop}^\square \to \cA^{\rm ord}_\partial({\rm FTop})$ for finite (possibly non discrete) topological spaces. Moreover, let $\Set\subset \Top$ be the category of small sets regarded as discrete spaces. We may consider the category of injections $Y\into X$ as a distinguished category $\Set^\square \subset\Set^{\bf 2}$. For $X\in \Set$ let $\epsilon_X : (Y, Z) \to (X, 1)\in \Set^\square$ such that $X= Y/Z$ is the push-out. Let $\cE_X=\{\epsilon_X\}$ be all such commutative squares and $\aleph = \{\cE_X\}_{X\in \Set}$. We get an homological subcategory $(\Set^\square, \aleph, I)$  regarding $\Set \subset \Top$ where homotopy of maps is equality.\\

c) Recall that the usual singular homology with $R$-coefficients $$H^{\rm Sing}(-;R) :\Top^\square \to\RMod$$ is an additive ordinary homology  with respect to the standard homological structure, \ie $H^{\rm Sing}(-;R)\in \Hom_{\rm Ord}(\Top^\square , \RMod)$ with coefficients $R = H_0^{\rm Sing}(1;R)$ (see \cite[\S 2.1]{Ha}, \cite{M} and \cite[Chap. VII]{ES}). This yields that $\cA^{\rm Ord}_\partial(\Top)$ is not trivial. Moreover, the exact $R$-linear realisation functor $r_{H^{\rm Sing}(-;R)}:\cA^{\rm Ord}_\partial(\Top)\to \RMod$ yields  $$\cA(H^{\rm Sing}(-;R))\subset \RMod$$ the category of topological Nori motives \ie the universal category generated by singular homology, which is a (non-full) non-trivial abelian subcategory. It is well known that singular homology is locally cellular with respect to the subcategory of cell complexes. Moreover, $H^{\rm Sing}_i(X ; E)$ singular homology with coefficients in a bundle yields an ordinary homology theory on the induced homological category (see \cite[Section 3.H]{Ha}). 
\end{remarks}

The category of simplicial sets $\SSet$ is endowed with a standard homological structure induced by the geometric realisation $|\ \ |: \SSet\to \Top$, where injective morphisms are distinguished,  $1 = \Delta^0$ is a point object and the simplicial interval $\Delta^1$ is such that $I = |\Delta^1|$, $\Xi\subset\SSet$ are all maps whose geometric realisation is a $(\Xi, I)$-generalised excision in such a way that 
$$|\  \ |: (\SSet^\square,\Xi,  \Delta^1)\to (\Top^\square,\Xi, I)$$
is an ordinary homological functor (as $|\  \ |$ preserves finite limits it preserves homotopies) which is additive (since $|\  \ |$ is a left adjoint), see Definition \ref{homdef}.

The standard cosimplicial space $\Im^n = |\Delta^n|$ for all $n\geq 0$ yields ${\rm Sing}^\Im_\d (-): \Top \to \SSet$,  the usual singular simplicial set functor. Consider the universal singular homology
$$H^{\rm Sing} : \Top^\square \to \iRMod$$ 
 as in Lemma \ref{relsing}, lifting the usual singular homology $H^{\rm Sing}(-;R)$ along the canonical exact functor 
 $\Ind r_R: \iRMod\to \RMod$ (by indization from \cite[Example 2.1.8]{UCTI}). 
\begin{lemma} \label{singord}
$H^{\rm Sing}\in \Hom_{\rm Ord}(\Top^\square,  \iRMod)$,  $\cA^{\rm Sing}_\partial(\Top)= \iRMod$ and $$r_{H^{\rm Sing}}: \cA^{\rm Ord}_\partial(\Top) \to \iRMod$$ is such that $r_{H^{\rm Sing}}!_\partial=id$. Moreover, for any Grothendieck category $\cA$ the singular homology 
$H^{\rm Sing}(-; A)$ with coefficients in an object $A\in\cA$ is ordinary with respect to  the standard homological structure. 
 \end{lemma}
 \begin{proof}
 The fact that $H^{\rm Sing}(-; |R|)$ satisfies homotopy invariance and excision with respect to $(\Top^\square,\Xi, I)$ follows by exactly the same arguments for the usual singular homology with $R$ coefficients. Therefore the Lemma \ref{relsing} and the Theorem \ref{unising} are equivalent and all singular homologies are ordinary. Since $H^{\rm Sing}_0(1)\df |R|\in \iRMod$ is the coefficient object and  $r_{H^{\rm Sing}}(H)= H^{\rm Sing}$, in particular $r_{H^{\rm Sing}}(\underline{R})=|R|$ henceforth $r_{H^{\rm Sing}}!_\partial= id : \iRMod\to \iRMod$.
 \end{proof}
  Moreover, $r_{H^{\rm Sing}}$ yields an interesting but mysterious abelian subcategory 
$$\bar{r}_{H^{\rm Sing}}: \cA (H^{\rm Sing})\df \cA^{\rm Ord}_\partial(\Top)/\ker r_{H^{\rm Sing}}\into \iRMod$$ such that $\iRMod\subseteq \cA (H^{\rm Sing})$ is a full subcategory.  

However, for $\Top$ a generic topological category, it seems difficult to describe $\cA^{\rm Ord}_\partial(\Top)$ and/or to see wether $\cA (H^{\rm Sing})$ is equal to $\iRMod$ or not, unless we restrict to ${\rm CW}$ the category of cell complexes.

\subsection{Cell complexes}
Consider $(\text{CW}, \Xi, I)$ the standard homological structure restricted to the subcategory $\text{CW}\subset \Top$. For CW-complexes say that sub-complexes are sub-distinguished. 
Recall that the geometric realisation $|\  \ |$  has essential image in ${\rm CW}$ and we easily obtain: 
\begin{propose} The functor $|\  \ |$  restricts to an equivalence
$$|\ \ |_\partial : \cA^{\Delta^1-\rm hi}_\partial({\rm Kan})\longby{\simeq}\cA^{I-\rm hi}_\partial({\rm CW})$$
where ${\rm Kan}\subset \SSet$ are the Kan complexes (note  that $\Delta^1\notin {\rm Kan}$).
\end{propose}
\begin{proof}
Actually, ${\rm Kan}/\approx_{\Delta^1}\cong {\rm CW}/\approx_I$ and any homotopy invariant homology factors uniquely through the quotient category in such a way that 
$$\Hom_{I-\rm hi} ({\rm CW}^\square, - )\cong \Hom_{\Delta^1-\rm hi} ({\rm Kan}^\square, - )$$
are equivalent 2-functors. 
\end{proof}
Restricting to Kan complexes ${\rm Kan}\subset \SSet$  we then obtain the following homological equivalence (see  Definition \ref{homeqdef})
$$|\ \ |_\partial : \cA^{\rm ord}_\partial({\rm Kan})\longby{\simeq}\cA^{\rm ord}_\partial({\rm CW})$$
and we have:
\begin{thm}\label{ordCW} The standard homological structure $({\rm CW},\Xi, I)$ is cellular, the functor
$$!_\partial : \iRMod\longby{\simeq}  \cA_\partial^{\rm Ord}({\rm CW})$$ is an equivalence and $H_i(X, Y)\cong !_\partial (H^{\rm Sing}_i(X,Y))\in \cA_\partial^{\rm Ord}({\rm CW})$ with quasi-inverse $r_{H^{\rm Sing}}$ so that $\cA (H^{\rm Sing})= \iRMod$. 
\end{thm}
\begin{proof}
For $X\in {\rm CW}$ the skeletal filtration $X_\d\to X$ is a good filtration for any additive ordinary homology $H : {\rm CW}^\square\to\cA$ in a Grothendieck category $\cA$ with coefficients $H_0(1)\df A\in \cA$  as it can be seen by a straightforward computation. In fact, note that for a CW pair $(X,Y)$ the quotient $X/Y$ is a pointed CW-complex and we have that $H_i(X, Y)\cong \tilde{H}(X/Y)$ by excision and homotopy invariance (\cf Definitions \ref{ES} - \ref{coeff}). Therefore, we have the following
$$\cdots \to \tilde{H}_i(Y)\to \tilde{H}_i(X)\to \tilde{H}_i(X/Y)\to \tilde{H}_{i-1}(Y)\to \cdots$$ 
long exact sequence; applying this to $X=D^n$, $Y=S^{n-1}$ and $X/Y=S^{n}$ for $n>0$ we obtain $\tilde{H}_i(S^{n})=0$ for $i\neq 0$ and $\tilde{H}_n(S^{n})=A$. Since $\coprod (D^p, S^{p-1})\to (X_{p},  X_{p-1})$ induces $\oplus H_i(D^p, S^{p-1})\cong H_i (X_{p},  X_{p-1})$ we obtain $H_i (X_{p},  X_{p-1})=0$ for $i\neq p$ and $H_p (X_{p},  X_{p-1})\cong \oplus A$ by additivity, where the sum is indexed on the set of $p$-cells.  If $X$ is finite dimensional we have that $X=X_d$ implies that $H_i(X_p)\cong H_i(X_d)$ for $i<p$ and since $X_0=\coprod *$ is a discrete set we have that $0= \oplus H_i(*)= H_i(X_p)$ if $i>p$ so that $H$ is locally cellular (see Definition \ref{ordcell}). If $X=\cup X_p$ is not finite dimensional, as explained in the proof of \cite[Lemma 2.34 (c)]{Ha} applied to our $H$, by additivity and the ``mapping telescope'' construction we still get that $H_i(X_p)\cong H_i(X)$ for $i <p$. Thus $H$ is locally cellular on ${\rm CW}$. Now as in Proposition \ref{cellularity} we obtain that $H_i(X)\cong H_i(C^H_*(X))$ can be computed by the cellular complex. By naturality of the complex in Lemma \ref{2=0} we get that this isomorphism is functorial for cellular mappings yielding that the relative homology $H$ on ${\rm CW}^\square$ is isomorphic to the relative homology determined by the cellular complex $C^H_*$ (using again that the homology of a ${\rm CW}$-pair $(X,Y)$ is given by that of $X/Y$ 
as a pointed ${\rm CW}$-complex). 

As a consequence:  universal additive ordinary homology $H:  {\rm CW}^\square\to\cA_\partial^{\rm Ord} ({\rm CW})$ with coefficients $H_0(1)\df \underline{R}= !_\partial (|R|)$ and singular homology $H^{\rm Sing}(-;\underline{R})$, which is additive and ordinary by Lemma \ref{singord}, are locally cellular. Therefore, we have that $H=H^{\rm Sing}(-;\underline{R})$ once we have checked that the two cellular complexes $C^H_*$ and $C^{H^{\rm Sing}}_*$ with coefficients in $\underline{R}$ do agree naturally. This latter checking can be done as in the proof of \cite[Thm.\,4.59]{Ha} (and, actually, it is equivalent to Eilenberg-Steenrod uniqueness theorem after an exact embedding of $\cA_\partial^{\rm Ord} ({\rm CW})$ in abelian groups, see also  \cite{M} and \cite[Chap.\, IV Thm.\, 10.1]{ES}).
In fact, to check that the cellular boundary maps are the same one is left to see (\cf the cellular boundary formula in the proof of \cite[Thm.\,2.35]{Ha}) that a map $S^n\to S^n$ of degree $m$ induces multiplication by $m$ on $H_n(S^{n})=\underline{R}=H^{\rm Sing}_n(S^{n})$ in both cases. This is clear since we have that both theories are additive with respect to the sum operation in $\pi_n(S^n)$ (\eg as explained in \cite[Lemma 4.60]{Ha}). In order to check naturality of the isomorphism $C^H_*=C^{H^{\rm Sing}}_*$ one is left with exactly the same problem as for the cellular boundary maps (see the argument in the proof of \cite[Thm.\,4.59]{Ha}).
Thus (as in Corollary \ref{ord=sing}) $$r_{H^{\rm Sing}}: \cA_\partial^{\rm Ord}({\rm CW})\longby{\simeq} \iRMod$$ 
yields an equivalence with quasi-inverse $!_\partial $ such that $!_\partial (H^{\rm Sing})= H^{\rm Sing}(-;\underline{R})=H$ with coefficients $\underline{R}$ and $|R|$ identified under the equivalence. Finally, since 
$\ker r_{H^{\rm Sing}} =0$ we have that $\cA (H^{\rm Sing}) =\iRMod$.
\end{proof}
Suppose that $R$ is coherent.  We then have that $r_R : \iRmod\onto \Rmod$ is a quotient and it is an equivalence if $R=K$ is a field, \cf \cite[Example 2.1.8]{UCTI}. Thus  $\Ind r_R: \iRMod\to \RMod$ is a quotient and from Theorem \ref{ordCW}  we obtain that the realization $$r_{H^{\rm Sing}(-;R)}=\Ind r_R:\cA^{\rm Ord}_\partial({\rm CW})\cong\iRMod \onto \RMod$$  is a quotient and it is an equivalence if $R=K$ is a field.
\begin{cor}\label{NoriCW}
Let $R$ be a coherent ring. We have that $\cA(H^{\rm Sing}(-;R))= \emph{\RMod}$, \ie topological Nori motives for CW-complexes are usual $R$-modules.
\end{cor}
\begin{remarks}
a) Consider finite and finite dimensional cellular complexes ${\rm CW}^{\rm fd}_{\rm fin}$ and singular homology with coefficients in $R$  Noetherian  
 $$H^{\rm Sing}(- ; R) : {\rm CW}^{\rm fd}_{\rm fin} \to \Rmod$$ taking values in the abelian category of finitely generated $R$-modules. From the same arguments in the proof of Theorem \ref{ordCW} we can see that $!_\partial : \iRmod\longby{\simeq} \cA_\partial^{\rm ord}({\rm CW}^{\rm fd}_{\rm fin})$ is an equivalence and the realization functor $r_{H^{\rm Sing}(- ; R)}: \cA_\partial^{\rm ord}({\rm CW}^{\rm fd}_{\rm fin})\onto \Rmod$ is a quotient.\\
 
b) Adding transfers for finite coverings, \ie $n$-sheeted covering spaces for some finite $n$, as in Remark \ref{trans}, recall that $H^{\rm Sing}(- ; R)$ is endowed with transfer homomorphisms. Thus, the previous quotient functor $r_{H^{\rm Sing}(- ; R)}$ factors through the functor $r_{H^\tr}: \cA^{\rm ord}_\partial({\rm CW}^{\rm fd}_{\rm fin})\cong \iRmod\to \cA^{\rm ord/\tr}_\partial({\rm CW}^{\rm fd}_{\rm fin})$. It appears reasonable to expect that $\cA^{\rm ord/\tr}_\partial({\rm CW}^{\rm fd}_{\rm fin})\cong \Rmod$ but a proof is missing.\\
\end{remarks}

\end{document}